\documentclass{article}

\usepackage[preprint]{neurips_2022}
\usepackage{graphicx}

\usepackage[utf8]{inputenc} 
\usepackage[T1]{fontenc}    
\usepackage[colorlinks=true,citecolor=blue]{hyperref}      
\usepackage{url}            
\usepackage{booktabs}       
\usepackage{amsfonts}       
\usepackage{nicefrac}       
\usepackage{microtype}      
\usepackage{xcolor}         

\usepackage{amsopn}

\usepackage{microtype}
\usepackage{subcaption}
\usepackage{booktabs}
\usepackage{dsfont}
\usepackage{xcolor}
\usepackage{natbib}

\usepackage{algorithm}
\usepackage[noend]{algorithmic}

\usepackage[mathcal]{eucal}

\usepackage{amssymb}
\usepackage{amsfonts}

\usepackage{amsmath}
\usepackage{dsfont}

\newcommand{\esp}[1]{\mathbb{E}\left[ #1 \right]}

\newcommand{\R}{\mathbb{R}}

\newcommand{\pcomm}{p_{\rm comm}}
\newcommand{\pcomp}{p_{\rm comp}}

\newcommand{\comm}{{\rm comm}}
\newcommand{\comp}{{\rm comp}}

\newcommand{\nfs}{m}
\newcommand{\nnodes}{n}
\newcommand{\ntokens}{K}
\newcommand{\token}{{\rm token}}
\newcommand{\ptilde}{\tilde{p}_{i,t}}
\newcommand{\proba}{p_i}
\newcommand{\lr}{\eta_{t}}
\newcommand{\deltaproba}{\delta_{i,t}}
\newcommand{\nbjumps}{N_{\rm jumps}}
\newcommand{\projA}{A^\dagger A}
\newcommand{\hnet}{{h}}
\newcommand{\hcomp}{{h_{\comp}}}
\newcommand{\hcomm}{{h_{\comm}}}
\newcommand{\Rp}{R_p}
\newcommand{\ncomps}{T_\comp}
\newcommand{\ncomms}{T_\comm}
\newcommand{\smooth}{\tilde{L}}
\newcommand{\sigmatilde}{\tilde{\sigma}}

\usepackage{amsthm}

\newtheorem{assumption}{Assumption}
\newtheorem{lemma}{Lemma}

\title{A principled framework for the design and analysis of token algorithms}

\author{%
	Hadrien Hendrikx \\ 
	EPFL\\ 
	\texttt{hadrien.hendrikx@epfl.ch}\\
}

\newtheorem{theorem}{Theorem}

\begin{document}

	\maketitle

	\begin{abstract}
		We consider a decentralized optimization problem, in which $n$ nodes collaborate to optimize a global objective function using local communications only. While many decentralized algorithms focus on \emph{gossip} communications (pairwise averaging), we consider a different scheme, in which a ``token'' that contains the current estimate of the model performs a random walk over the network, and updates its model using the local model of the node it is at. Indeed, token algorithms generally benefit from improved communication efficiency and privacy guarantees. We frame the token algorithm as a randomized gossip algorithm on a conceptual graph, which allows us to prove a series of convergence results for variance-reduced and accelerated token algorithms for the complete graph. We also extend these results to the case of multiple tokens by extending the conceptual graph, and to general graphs by tweaking the communication procedure. The reduction from token to well-studied gossip algorithms leads to tight rates for many token algorithms, and we illustrate their performance empirically. 
	\end{abstract}

	\section{Introduction}
	
	Modern machine learning relies on increasingly large models that train on increasingly large datasets: distributed optimization is thus crucial to scaling the training process. In the centralized paradigm, at the heart of \emph{Federated Learning}~\citep{kairouz2021advances}, the system relies on a server that aggregates models and gradients, and manages the nodes. Although quite efficient, this setting has several drawbacks: (i) nodes need to trust the server enough to send it sensitive data, (ii) there is a communication bottleneck at the server, which limits scaling and (iii) training stops if the server fails.  
	
	In the \emph{decentralized} setting~\citep{boyd2006randomized,lopes2007incremental,shi2015extra,nedic2017achieving}, nodes are linked by a communication graph, and directly communicate with their neighbours in this graph instead of a central coordinator. This allows for better scaling, and is also more robust since the server is no longer the single point of failure. Yet, due to the lack of coordination, decentralized algorithms often require many peer-to-peer communications compared to centralized ones, and a gain in privacy is not always guaranteed. 
	
	\emph{Token}, or random-walk algorithms~\citep{bertsekas1997new,ram2009incremental,johansson2010randomized,shah2018linearly,mao2020walkman}, work in the following way: a token owns an estimate of the model, ``walks'' over the graph, and sequentially visits nodes. When the token is held by a node, it updates its model, either by computing a gradient using the node's data, or by using the local model of the node. Then, the token is transmitted (or ``jumps'') to a new node.
	
	Some instanciations of these algorithms can be seen as a middle-point between centralized and decentralized algorithms. Indeed, the token plays the role of a server, since it owns the global model and receives updates from nodes. Yet, the token is no longer attached to a physical node as in the centralized case, but rather exchanged between computing nodes in a decentralized way.
	
	Besides, unlike standard centralized algorithms, each node may maintain a local parameter and update it using local updates. In that sense, some token algorithms (such as the ones developed in this work) closely resemble local methods, that are very popular in federated learning~\citep{mcmahan2017communication,stich2018local,lin2018don}. The main difference is that instead of exchanging information through periodic exact averaging or gossip steps, communication is ensured through the roaming token. This allows to easily adapt the algorithms to the features of the system, by making either more local steps or more communication steps. 
	
	\subsection{Related work}
	
	Many early works study token (or random-walk) algorithms~\citep{bertsekas1997new,nedic2001incremental,ram2009incremental,johansson2010randomized}. Yet, they focus on stochastic (sub)gradients algorithms, and thus lack linear convergence guarantees. The recent literature on token algorithms can be divided into two main lines of work that reflect the two main strengths of token algorithms: \textit{communication efficiency} and \textit{privacy preservation}.  
	
	\textbf{Communication efficiency.} \citet{mao2020walkman} introduce \emph{Walkman}, a token algorithm based on an augmented Lagrangian method. Walkman works for general graphs and is shown to be communication-efficient provided graphs are well-connected enough. Yet, it only obtains linear convergence on least squares problems. When Walkman uses gradients (instead of proximal operators), it requires a step-size inversely proportional to the square of the smoothness constant of the problem, which is impractical. Variants of Walkman guarantee communication efficiency when walking over Hamiltonian cycles~\citep{mao2018walk}. \citet{balthazar2020distributed} consider the problem of distributed linear estimation, and use a token algorithm to aggregate the measurements of all nodes.
	
	\textbf{Multiple tokens.} When a single token walks the graph, there are no parallel communications. A natural fix to speed up algorithms is to allow multiple tokens to walk the graph in parallel, as recently done by~\citet{chen2022asynchronous}, whose approach is also based on an augmented Lagrangian method.
	
	\textbf{Privacy Preservation.} The favorable privacy guarantees claimed by decentralized algorithms are actually mainly proven for token algorithms. For instance, the Walkman algorithm presented above has also been extended to guarantee privacy preservation~\citep{ye2020incremental}. Besides,~\citet{cyffers2022privacy} show that token algorithms satisfy a relaxation of local differential privacy, and match the guarantees offered by a trusted central server. They give a simple algorithm for ring and complete topologies. Similarly, \citet{bellet2020started} study the privacy guarantees of a rumour spreading algorithm, and show that a single token spreading the rumour is optimal, while multiple tokens achieve optimal trade-offs between privacy and speed. In this work, we focus on the convergence guarantees of token algorithms, and leave the privacy preservation guarantees for future work.
	
	\textbf{Dual Decentralized algorithms.} Our framework is based on applying the dual approach for decentralized algorithms~\citep{jakovetic2014linear,boyd2011distributed} to the analysis of token algorithms. This dual approach leads to very fast algorithms, and in particular \citet{scaman2017optimal,uribe2020dual} used it to develop optimal decentralized algorithms. Then,~\citet{hendrikx2018accelerated} showed that it can also be used to accelerate randomized gossip, and used an augmented graph formulation to obtain decentralized variance-reduced extensions~\citep{hendrikx2019accelerated,hendrikx2020optimal,hendrikx2020dual}.
	
	\subsection{Our contributions}
	As discussed in the previous section, token algorithms are still rare, and very few algorithms offer linear convergence guarantees, let alone integrating more advanced optimization tricks such as variance reduction or acceleration. Besides, most of the literature focuses on only one token, which is communication-efficient but very slow. In this work, we pave the way for the design and analysis of new efficient token algorithms, and in particular we: 
	\begin{enumerate}
		\item Introduce a general framework for designing and analyzing token algorithms.
		\item Give a simple algorithm with linear convergence guarantees on complete graphs that match those of both centralized and decentralized (gossip) optimization. 
		\item Speed up this simple algorithm by using multiple tokens.
		\item Leverage the general framework to analyze variants of the simple token algorithm, such as stochastic gradients with variance reduction and acceleration. 
		\item Extend the convergence results to general graphs, by tweaking the communication protocol.
	\end{enumerate}
	The general framework is based on the dual approach for decentralized algorithms~\citep{jakovetic2014linear,scaman2017optimal}, and in particular~\citet{hendrikx2020dual}, and so the algorithmic core similar, namely Bregman coordinate descent (with some adaptations) for the simple and variance-reduced algorithms, and Accelerated Proximal Coordinate Descent~\citep{lin2015accelerated} for the accelerated one. 
	
	\section{Conceptual graph approach for token algorithms.}
	
	\begin{figure}
		\centering
		\includegraphics[width=0.8\linewidth]{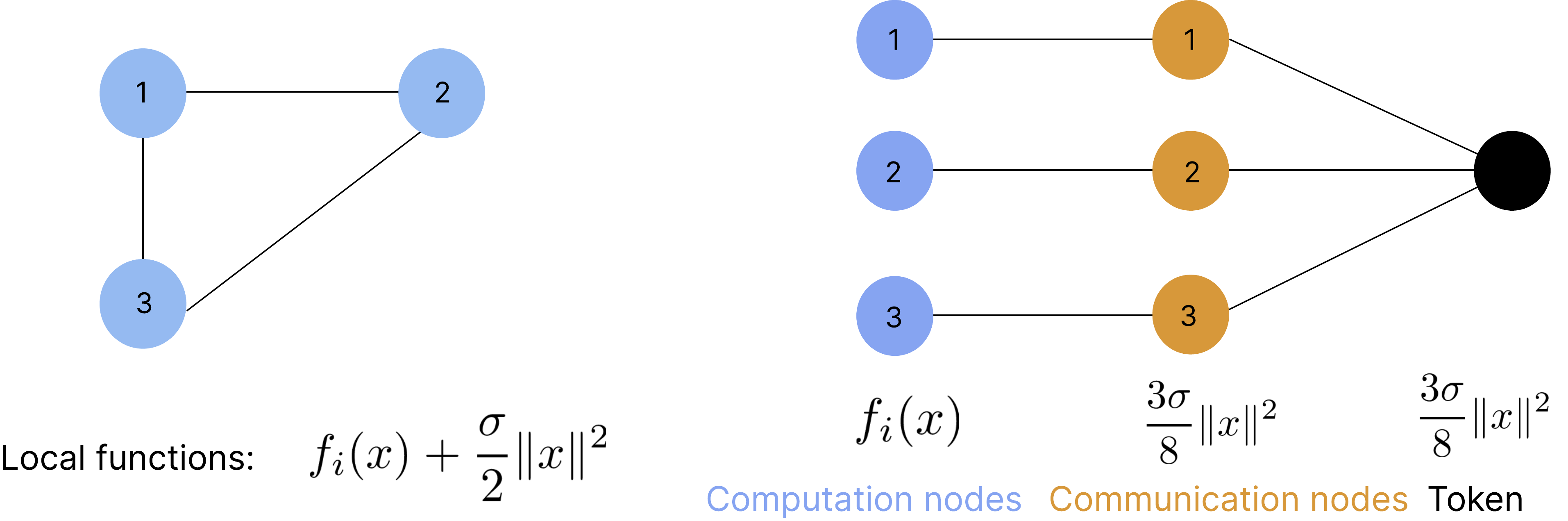}
		\caption{Left: base communication graph. Right: Conceptual graph, with modified local objectives.}
		\label{fig:simple_token}
	\end{figure}

	\subsection{Building the conceptual graph}
	We consider the following distributed problem, where each $f_i$ is a local function at node $i$:
	\begin{equation} \label{eq:prob_main}
	\min_{\theta \in \R^{d}} \sum_{i=1}^n \left[f_i(\theta) + \frac{\sigma}{2}\|\theta\|^2\right].
	\end{equation}
	We assume that each $f_i$ is convex and $L$-smooth over $\R^d$, which writes if $f_i$ is twice differentiable as $0 \preccurlyeq \nabla^2 f_i(x) \preccurlyeq L \ I_d$, where $I_d \in \R^{d \times d}$ is the identity matrix of dimension $d$. The \emph{condition number} of this problem is $\kappa =  1 + L/\sigma$.
	The key idea of this paper is to reduce the analysis of token algorithms to that of standard decentralized gossip algorithms on conceptual graphs. We follow the \emph{dual approach} for building decentralized optimization algorithms, and rewrite Problem~\eqref{eq:prob_main} as:
	\begin{equation}\label{eq:simple_token}
	\underset{\forall i, \ \theta^{(i)} = u^{(i)}, \text{ and } u^{(i)} = v}{\min_{\theta \in \R^{n \times d}, \ u \in \R^{n\times d}, \ v \in \R^d,}} \sum_{i=1} f_i(\theta^{(i)}) + \frac{n\sigma}{2(n+1)} \|u^{(i)}\|^2 + \frac{n\sigma}{2(n+1)}\|v\|^2.
	\end{equation}
	To write this reformulation, we have applied the consensus constraints (equality constraints for neighbours) given by the conceptual graph represented in Figure~\ref{fig:simple_token} (right). To build this conceptual graph, we add a conceptual node (with its own parameter) corresponding to the token, with local objective $\nnodes \sigma \| \cdot\|^2 / 2(\nnodes+1)$, and we split all local nodes into a computation part (that contains $f_i$), and a communication part (that contains $\nnodes\sigma \| \cdot\|^2 / \nnodes(n+1)$. Node that the total regularization weight is still $\nnodes\sigma/2$. Then, all computation nodes are linked to their respective communication nodes, which are themselves linked to the token. Splitting each node between communication and computation part has two benefit: (i) it allows us to use the dual-free trick from~\citet{hendrikx2020dual}, and obtain primal updates despite the dual approach, and (ii) it allows to decouple communications and computations. In particular, nodes can perform local steps even when they don't hold the token.
	
	Now that we have defined the framework, it is important to make sure that this corresponds to a token algorithm. Updating the edge between the token and node $i$ at time $t$ in the conceptual graph means that the token jumped to node $i$ at team $t$. Thus, to ensure the token aspect, we must enforce that if the edge between the token and node $i$ is updated at time $t$, and the edge between the token and node $j$ is updated at time $t+1$, then node $j$ has to be a neighbour of node $i$ (since the token jumped from $i$ to $j$ at time $t+1$). We apply the dual approach to Problem~\eqref{eq:simple_token}, which is inherited from the conceptual graph, but \emph{the sampling of the edges is ruled by the actual communication graph}. In a complete graph, this does not impose any additional constraints, and this is why our convergence results are initially derived in this setting. In arbitrary graphs, this means that the sampling of the edges (and so the coordinate descent algorithm applied to the dual formulation) must follow a Markov Chain, which leads to considerably harder analyses. In Section~\ref{sec:general_graphs}, we present a trick to circumvent this difficulty, which consists in not performing the update step every time the token jumps to a new node.
	
	\begin{figure}
		\centering
		\includegraphics[width=0.9\linewidth]{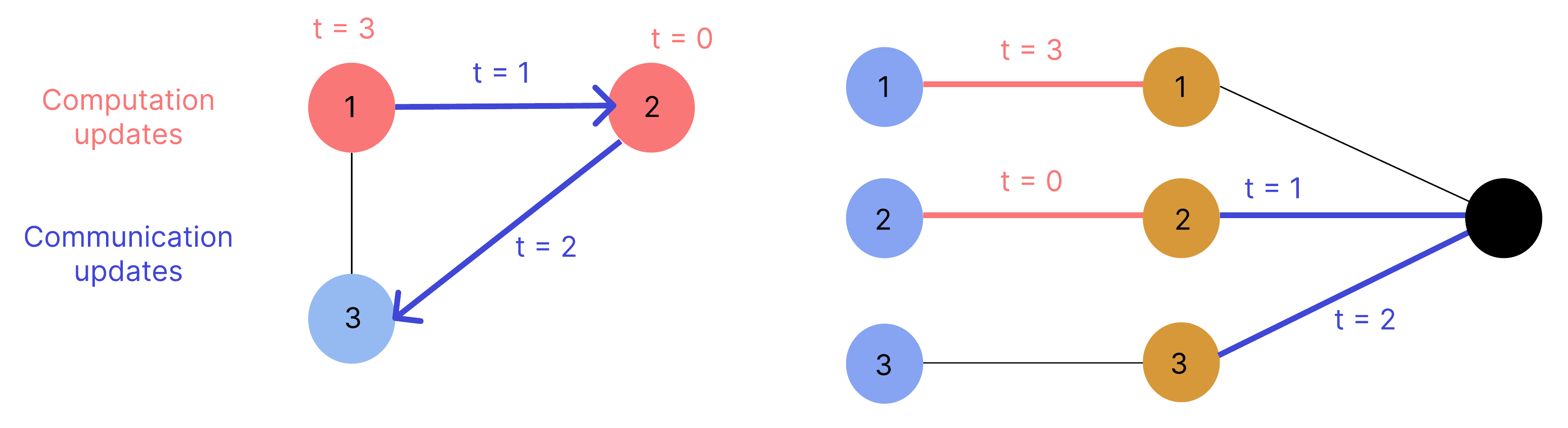}
		\caption{An example execution of the token algorithm on the base graph (left), and the corresponding edges updated in the conceptual graph (right). The sequence is: $t=0$: local computation at node $2$, $t=1$: the token jumps to node $2$, $t=2$: the token jumps to node $3$, $t=4$: local computation at node $0$. Note that the updates at $t=2$ and $t=3$ can actually be performed in parallel since they affect different nodes, and that the token updates its estimate with the node it arrives at after each jump.}
		\label{fig:token_example}
	\end{figure}
	
	We will now show that this conceptual graph view allows to efficiently design fast algorithms, by making clear links with dual approaches for decentralized optimization. This will prove especially useful in the next section, when we introduce multiple tokens, variance-reduction and acceleration. In the basic case (one token, full gradients), Equation~\eqref{eq:simple_token} resembles the consensus formulation of Walkman~\citep{mao2020walkman}, which is also obtained through a (primal-)dual formulation. Yet, we split each node into two subnodes to allow for local steps, and we can then harness the power of the dual approach for decentralized optimization to significantly improve the base algorithm, as done in Section~\ref{sec:additional_tricks}. 
	
	\subsection{Deriving the single token algorithm}
	Following~\citet{hendrikx2020dual}, we take a dual formulation of Problem~\eqref{eq:simple_token}, and apply Bregman block coordinate descent to obtain the simple token algorithm, which corresponds to Algorithm~\ref{algo:simple_token} with $\ntokens=1$. This leads to dual-free updates~\citep{lan2017optimal}, which are simple to implement. Yet, in~\citet{hendrikx2020dual}, all communication edges are sampled at once, which would mean that the token receives updates from all nodes at the same time. This is not possible in our case, so we adapted the Bregman block coordinate descent algorithm to better fit the structure of Problem~\eqref{eq:simple_token}, as detailed in Appendix~\ref{app:bcd}. 
	
	\begin{algorithm}
		\caption{Token Gradient Descent$(z_0)$}
		\label{algo:simple_token}
		\begin{algorithmic}[1]
			\STATE $\tilde{\sigma} = \frac{\nnodes}{\nnodes+\ntokens} \sigma$, $\alpha = \frac{2 \ntokens}{L}$, $\eta = \min\left( \frac{\tilde{\sigma} p_\comm}{2\nnodes \ntokens}, \frac{\pcomp}{\nnodes\alpha(1 + L / \tilde{\sigma})} \right)$, $\rho_\comm = \frac{\nnodes \ntokens \eta}{p_\comm\tilde{\sigma}}$, $\rho_{\comp} = \frac{\nnodes \alpha \eta}{\pcomp}$. \COMMENT{Init}
			\STATE $\forall i \in [\nnodes]$, $\theta_0^{(i)} = - \nabla f_{i}(z_0^{(i)}) / \tilde{\sigma}$;   \ \ $\forall k \in [\ntokens]$, $\theta_0^{\token, k} = 0$ . \COMMENT{$z_0$ is arbitrary but not $\theta_0$.}
			\FOR[Run for $T$ iterations]{$t=0$ to $T-1$}
			\IF{ communication step (with probability $p_\comm$)}
			\STATE Pick $i \sim \mathcal{U}([n]), \ k \sim \mathcal{U}([\ntokens])$ \COMMENT{Choose next node and token uniformly at random}
			\STATE $\theta_{t+1}^{\token, k} = \theta_t^{\token, k} - \rho_{\comm}  (\theta_t^{\token, k} - \theta_t^{(i)})$ \COMMENT{Token update}
			\STATE $\theta_{t+1}^{(i)} = \theta_t^{(i)} + \rho_{\comm}(\theta_t^{\token, k} - \theta_t^{(i)})$ \COMMENT{Local node update}
			\ELSE 
			\STATE Pick $i \sim \mathcal{U}([n])$ \COMMENT{Choose one node at random}
			\STATE $z_{t+1}^{(i)} = \left(1 - \rho_{\comp}\right)z_t^{(i)} + \rho_{\comp} \theta_t^{(i)}$ \COMMENT{Virtual node update}
			\STATE $\theta_{t+1}^{(i)} = \theta_t^{(i)} - \frac{1}{\tilde{\sigma}}\left(\nabla f_{i}(z_{t+1}^{(i)}) - \nabla f_{i}(z_{t}^{(i)})\right)$ \COMMENT{Local update using $f_{i}$}
			\ENDIF
			\ENDFOR
			\STATE \textbf{return} $\theta_K$
		\end{algorithmic}
	\end{algorithm}
	
	\begin{theorem}[Token algorithm] \label{thm:simple_token}
		For $\varepsilon > 0$, the number of steps required by Algorithm~\ref{algo:simple_token} with a single token ($\ntokens = 1$) and $p_\comp = p_\comm = \frac{1}{2}$ to reach error $\| \theta_t - \theta_\star\|^2 \leq \varepsilon$ is of order:
		\begin{equation}
		\ncomps = O\left(\kappa \log \varepsilon^{-1}\right) \ \text{ and } \ \ncomms = O\left(n \kappa \log \varepsilon^{-1}\right),
		\end{equation}
		where $\ncomps$ is the expected number of gradient steps performed by each node, and $\ncomms$ is the expected number of communication updates (jumps) performed by the token. 
	\end{theorem}
	
	\begin{proof}[Proof sketch.]
		This result follows from the guarantees of Bregman Coordinate Descent applied to a dual reformulation of Problem~\eqref{eq:simple_token}. We need to evaluate the relative strong convexity and directional smoothness constants of the problem, and link them with spectral properties of the conceptual graph, as well as the regularity of the local functions. Details can be found in Appendix~\ref{app:simple_token}.
	\end{proof}
	
	\textbf{Implementation.} Algorithm~\ref{algo:simple_token} requires sampling updates uniformly at random over the whole system. This can be implemented in a decentralized fashion by sharing a random seed between all nodes. An alternative is that all nodes wake up and perform updates following a Poisson point process.   
	
	\textbf{Communication complexity.} The total number of communications is of order $O(n \kappa)$. This matches the complexity of centralized algorithms, in which the server communicates once with each node at each round, and there are $O(\kappa)$ rounds in total. 
	Yet, in terms of time, the $O(n)$ centralized communications can take place in parallel provided there is enough bandwidth, whereas when $\ntokens=1$, the $O(n)$ communications from Algorithm~\ref{algo:simple_token} must be sequential. 
	
	\textbf{Computation complexity.} In average, each node performs $O(\kappa)$ local computations, 
	which is the ``right'' complexity for non-accelerated algorithms. Algorithm~\ref{algo:simple_token} is as computationally efficient as standard centralized or decentralized algorithms in this sense. Besides, unlike token exchanges that need to be sequential, nodes can compute local gradient updates in parallel. Therefore, the computation time of Algorithm~\ref{algo:simple_token} matches the centralized time. Instead, when using gradients, Walkman~\citep[Theorem 1]{mao2020walkman} requires a step-size proportional to $L^{-2}$ just for convergence, which would lead to a significantly worse computation complexity of at least $O(L\kappa)$. 
	
	\textbf{Sampling variants.} In Algorithm~\ref{algo:simple_token}, one computation update corresponds to one node performing a gradient update, without any communication involved. Note that by changing the sampling of the dual coordinates, it is possible to design other algorithms algorithms. For instance, we can choose to have nodes perform a local computation only when they receive the token, similarly to Walkman. This would yield a similar rate as the one in Theorem~\ref{thm:simple_token}, but does not allow for local steps. It would thus not be possible for instance to perform a lot of fast communications, while slow computations take place in parallel. 
	Yet, our framework can also handle this variant (and many others), including with the tricks developed in the next section.

	\section{Extensions of the single token algorithm}
	\label{sec:additional_tricks}
	In the previous section, we have introduced the conceptual graph, and showed how it allows to leverage the existing tools from (dual) decentralized optimization to analyze a simple yet already efficient token algorithm. We now demonstrate the flexibility and generality of this framework by introducing, analyzing and combining three important variants of the token algorithm: multiple tokens, variance reduction, and acceleration.  
	
	\subsection{The case of several tokens}
	When there is a single token walking on the graph, resources are used in an efficient way, and privacy guarantees are strong, but mixing is very slow (up to $n$ times slower in a complete graph for instance). This is due to the fact that there are no parallel communications. One natural solution is to use multiple tokens that walk the graph in parallel. Yet, this is generally harder to analyze and, to the best of our knowledge, there only limited theory on multi-token algorithms, with convergence rates that show actual improvement over single tokens. Our conceptual graph framework allows us to directly extend Theorem~\ref{thm:simple_token} to the case of multiple tokens. 
	
	\begin{figure}
		\centering
		\includegraphics[width=0.5\linewidth]{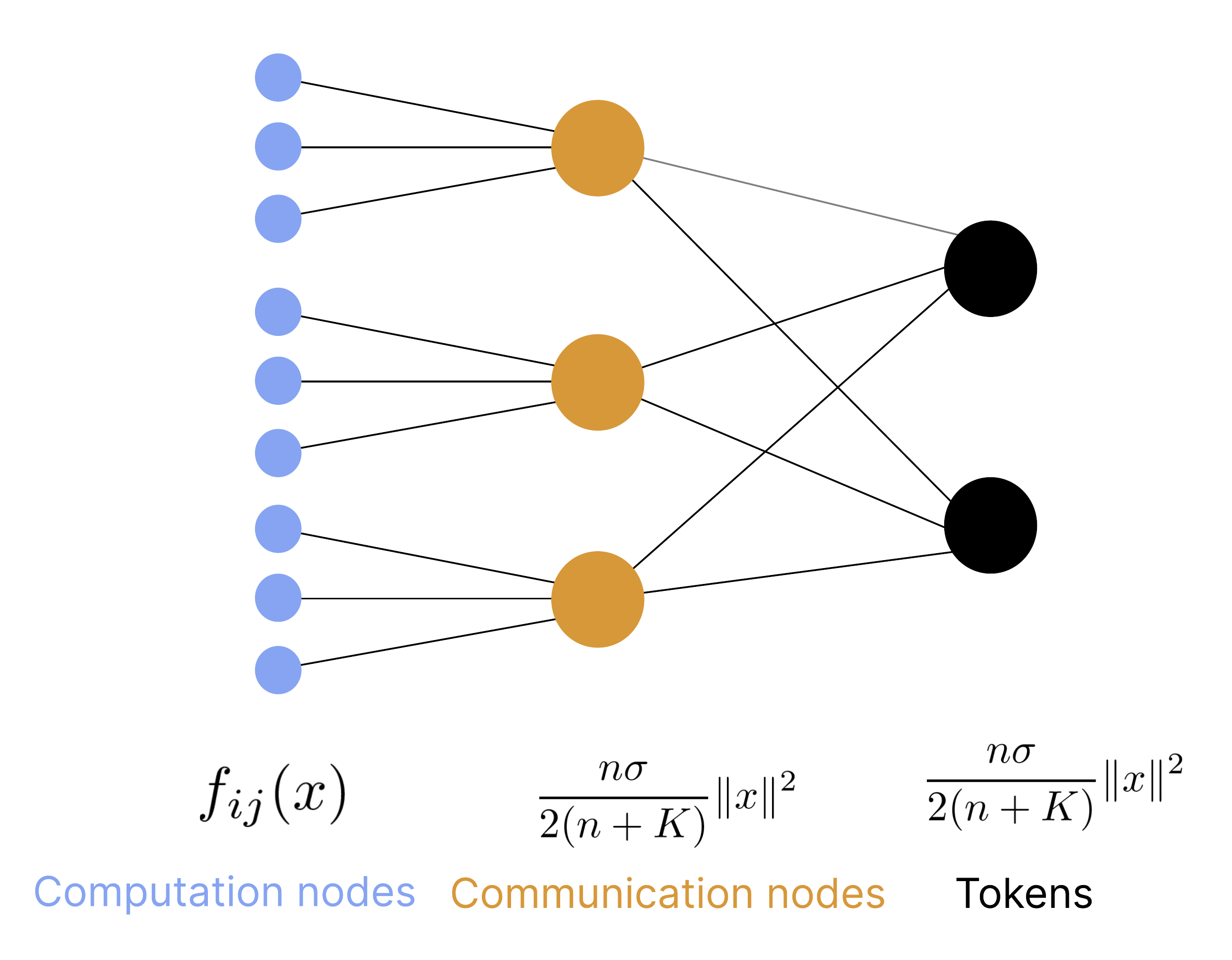}
		\caption{Conceptual graph of size $\nnodes = 3$ with finite-sum local objectives ($\nfs = 3)$ and multiple tokens ($\ntokens = 2$).}
		\label{fig:multiple_tokens}
	\end{figure}
	
	To do so, we build a different conceptual graph. Namely, we add one new node for each token, and link all ``token nodes'' to the actual nodes of the network, as shown in the right part of Figure~\ref{fig:multiple_tokens}. Then, we apply the dual approach to this new conceptual graph, which has a different topology but which we know how to handle. This is how we obtain Algorithm~\ref{algo:simple_token} for the general case of $K \geq 1$. 
	
	\begin{theorem}[Multiple tokens] \label{thm:multiple_tokens}
		For $\varepsilon > 0$, the number of steps required by Algorithm~\ref{algo:simple_token} with $1 \leq \ntokens \leq \nnodes$ and $p_\comp = p_\comm = \frac{1}{2}$ to reach error $\| \theta_t^{{\rm token}, k} - \theta_\star\|^2 \leq \varepsilon$ is of order: 
		\begin{equation}
		\ncomps = O\left(\kappa \log \varepsilon^{-1}\right) \text{ and } \ncomms = O\left(\frac{\nnodes}{\ntokens} \kappa \log \varepsilon^{-1}\right),
		\end{equation}
		where $\ncomps$ is the expected number of gradient steps performed by each node, and $\ncomms$ is the number of jumps performed \emph{per token}. In particular, the total communication complexity is the same as in Algorithm~\ref{algo:simple_token}, but now the burden is shared by $\ntokens$ tokens that walk the graph in parallel.
	\end{theorem}
	
	
	\textbf{Token interactions.} In the formulation inherited from Figure~\ref{fig:multiple_tokens}, tokens interact with nodes, but not between themselves. We can change the formulation to add interactions between tokens by adding edges between them in the conceptual graph. This would mean that the tokens would mix information when they meet. Yet, this would only marginally increase the connectivity of the conceptual graph, and would thus not speedup the algorithm by more than constant factors. 
	
	\subsection{Variance reduction in the finite sum case}
	We have seen that changing the conceptual graph on which the dual formulation is applied changes the resulting token algorithm. In the previous section, we used this to speed-up communications by having several tokens walk the graph in parallel. We now leverage it to speed-up computations by avoiding local full gradient computations at each node. We now assume that each local objective writes $f_i(x) = \sum_{j=1}^\nfs f_{ij}(x)$. In this case, each full gradient requires $\nfs$ stochastic gradient $\nabla f_{ij}$ computations, so Algorithm~\ref{algo:simple_token} requires $O(\nfs \kappa)$ stochastic gradient in total. Instead, \emph{variance reduction} techniques~\citep{schmidt2017minimizing,johnson2013accelerating,defazio2014saga,shalev2016sdca} only require $\nfs + \kappa_s$ stochastic gradients, where $\kappa_s = \sum_{j=1}^\nfs (1 + L_{ij} / \sigma)$, where $L_{ij}$ is the smoothness of function $f_{ij}$. Although $\nfs \kappa = \kappa_s$ in the worst case (where the $\nabla f_{ij}$ are all orthogonal), $\kappa_s$ is generally smaller than $m\kappa$, leading to the practical superiority of finite-sum methods.
	
	In our case, we introduce the finite-sum aspect by combining the conceptual graph with an augmented graph formulation~\citep{hendrikx2019accelerated,hendrikx2020dual}. Instead of splitting each node into 2 parts, containing respectively $f_i$ and the regularization part, as in Figure~\ref{fig:simple_token}, we split it into a star sub-network, with each $f_{ij}$ linked to the regularization part, as shown if the left part of Figure~\ref{fig:multiple_tokens}. This new conceptual graph leads to a new algorithm, \textbf{T}oken \textbf{V}ariance \textbf{R}eduction (TVR), that has the following convergence guarantees:
	
	\begin{theorem}[Variance Reduction] \label{thm:token_vr}
		For $\varepsilon > 0$, the number of steps required by TVR with $1 \leq \ntokens \leq \nnodes$ and $\pcomp = \frac{\kappa_s}{m - 1 + \kappa_s}$ to reach error $\| \theta_t - \theta_\star\|^2 \leq \varepsilon$ is of order: 
		\begin{equation}
		\ncomps = O\left((\nfs + \kappa_s) \log \varepsilon^{-1}\right) \ \text{ and } \ \ncomms = O\left(\frac{n}{\ntokens} \kappa_s \log \varepsilon^{-1}\right),
		\end{equation}
		where $\ncomps$ is the expected number of stochastic gradient steps performed by each node, and $\ncomms$ the number of jumps performed \emph{per token}. Compared to Theorem~\ref{thm:multiple_tokens}, the computation complexity goes from $\nfs \kappa$ stochastic gradients ($\kappa$ full gradients) to $\nfs + \kappa_s$, which is generally much smaller. 
	\end{theorem}
	
	TVR performs the same communication steps as Algorithm~\ref{algo:simple_token}, with slightly different computation steps, adapted for the stochastic case: there are now $\nfs$ functions $f_{ij}$ (and so parameters $z^{(ij)}$), instead of just one. The full algorithm and the proof of Theorem~\ref{thm:token_vr} are detailed in Appendix~\ref{app:tvr}.
	
	\subsection{Acceleration}
	We have seen that the conceptual graph view of token algorithms allows to naturally extend the simple single-token batch algorithm to a multi-token variance-reduced algorithm. We now show that by applying a different optimization algorithm to the same dual formulation, we obtain an accelerated algorithm from the same framework. We refer the reader to Appendix~\ref{app:acceleration} for details and derivations.
	
	\begin{theorem}[Token Accelerated Variance-Reduced]\label{thm:tavr}
		For $\varepsilon > 0$, the number of steps required by Accelerated TVR with $1 \leq \ntokens \leq \nnodes$ to reach error $\| \theta_t^{{\rm token}, k} - \theta_\star\|^2 \leq \varepsilon$ is of order:
		\begin{equation}
		\ncomps = O\left((\nfs + \sqrt{\nfs\kappa_s}) \log \varepsilon^{-1}\right) \text{ and } \ncomms = O\left(\frac{\nnodes}{\ntokens} \sqrt{\kappa_s} \log \varepsilon^{-1}\right)
		\end{equation}
		In particular, the dependences on the objective regularity are replaced by their accelerated versions. This matches the optimal complexities from~\citet{hendrikx2020optimal}. 
	\end{theorem}
	
	\textbf{Proximal oracles.} This algorithm is based on Accelerated Proximal Coordinate Gradient~\citep{lin2015accelerated,hendrikx2019accelerated}, and thus uses proximal operators of the functions $f_{ij}$ instead of gradients. Yet, this is also the case of Walkman, and it is quite cheap in case the $f_{ij}$ are generalized linear model of the form $f_{ij}(\theta) = \ell(x_{ij}^\top \theta)$, where $\ell: \R \mapsto \R$. 
	
	\textbf{Continuized framework.} TAVR requires each node to perform local convex combinations at each step. This introduces global synchronization constraints, that we can get rid of using a continuized version of the algorithm~\citep{even2021continuized}.  
	
	\subsection{General graphs}
	\label{app:general_graphs}
	
	\label{sec:general_graphs}
	All the results presented so far are for the complete communication graph, meaning that the token can directly jump from any node to any other, allowing to prove strong convergence rates. We now show how to extend these results to general graphs. To do so, we analyze a slightly different communication procedure: instead of just one jump, each token performs $\nbjumps$ jumps before averaging with the node it lands at according to Algorithm~\ref{algo:simple_token} (lines 6-7). If enough steps are taken, and the underlying Markov Chain is irreducible and aperiodic, then the probability that the token lands at node $j$ from node $i$ after $\nbjumps$ steps is approximately equal to $\pi_\star(j)$, where $\pi_\star$ is the stationary distribution of a random walk over the communication graph. In particular, by allowing multiple steps before performing the actual communication update, all nodes can be reached from all nodes, as in the complete graph case. The actual sampling probabilities $\ptilde$ depend on the node at which the token is at time $t$. Yet, if $\ptilde$ is close enough to $\pi_\star(i)$, we can just adapt the step-sizes in coordinate descent, and obtain convergence regardless. Details can be found in Appendix~\ref{app:general_graphs}
	
	\begin{theorem}
		Assume that matrix $W \in \R^{\nnodes \times \nnodes}$ defining the token transitions is such that for any $\pi_0$, $\|W^t \pi_0 - \pi_\star\|_\infty \leq C (1 - \gamma)^t$. Then, if the token jumps $O(\gamma^{-1} \log(C / \eta \mu))$ times before performing an averaging step, the communication complexity of Algorithm~\ref{algo:simple_token} (ignoring log factors) is $O(\frac{n\kappa}{\gamma} \log \varepsilon^{-1})$. Theorem~\ref{thm:token_vr} can be adapated in the same way.
	\end{theorem}
	
	Note that these rates are for uniform stationary distribution $\pi_\star(i) = \nnodes^{-1}$ for all $i\in [n]$. Rates for non-uniform stationary distributions can also be obtained, and would depend on $\max_i \pi_i / L_i$, so that the algorithm is still fast if nodes that are visited less frequently have better smoothness.
	
	\textbf{Comparison with existing decentralized algorithms.} Note that decentralized algorithms such as EXTRA~\citep{shi2015extra,li2020revisiting} require a total of $O(E(\kappa + \gamma^{-1}) \log \varepsilon^{-1})$ communications, where $E$ is the number of edges in the graph. In particular, our token algorithm is more communication efficient even in general graphs as long as either $\kappa$ or $\gamma^{-1}$ is small compared to $E/n$, the average node degree. Note that this complexity is also better than that of Walkman~\citep{mao2020walkman}, which depends on $\gamma^2$ instead of $\gamma$, besides proving linear convergence only in limited settings. 
	
	\textbf{Directed graphs time-varying graphs.} With the communication-skipping variant, all that matters is that the probability of being at node $j$ after taking $O(\gamma^{-1} \log(\eta\mu))$ steps from node $i$ is close to some $\pi_\star(j) > 0$. In particular, this does not imply the reversibility of the Markov Chain used for communications (and so can be used for directed graphs), or even its stationarity.
	
	\section{Experiments}
	In the previous section, we leveraged the conceptual graph framework to design and analyze several (multi-)token algorithms. We now illustrate their differences in Figure~\ref{fig:token_exp}. We use a step-size $\beta = 1/L$ for Walkman, which is much higher than the one from~\citet{mao2020walkman}. The communication complexity is the total number of token jumps (regardless of the number of tokens), and the computation complexity is the total number of gradients (on single $f_{ij}$) computed. Time is obtained by setting $\tau_\comp = 1$ for computing one individual gradient and $\tau_\comm = 10^3$ for one communication, so that one communication is faster than computing one full local gradient. For gradient descent (GD), we present 2 variants of the allreduce protocol: \emph{all-to-all} , which has a high number of communications $\nnodes(\nnodes - 1)$ but small time ($1$) per step, and \emph{ring} (sequentially averaging over a directed ring), which has a small number of communication($2\nnodes$) but high time  ($2\nnodes$) per step. Note that a fully centralized implementation would get both small communication complexity ($2\nnodes$) and time ($2$). \emph{Token} and \emph{TVR} respectively refer to the algorithms analyzed in Theorems~\ref{thm:multiple_tokens} and~\ref{thm:token_vr}. EXTRA is a standard decentralized algorithm~\citep{shi2015extra}. Additional details can be found in Appendix~\ref{app:experiments}.
	
	\begin{figure}
		
		\centering
		\includegraphics[width=0.45\linewidth]{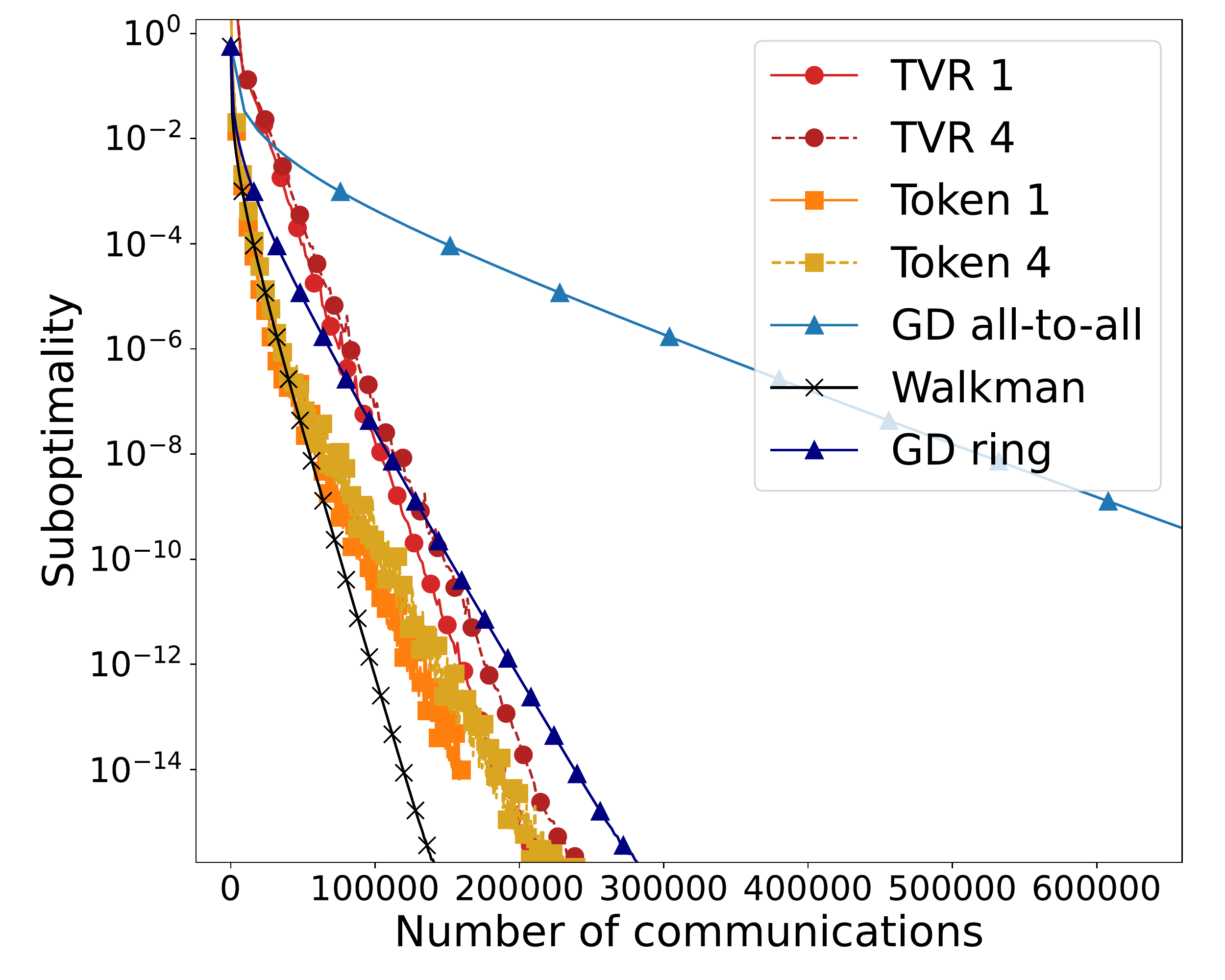}
		\includegraphics[width=0.45\linewidth]{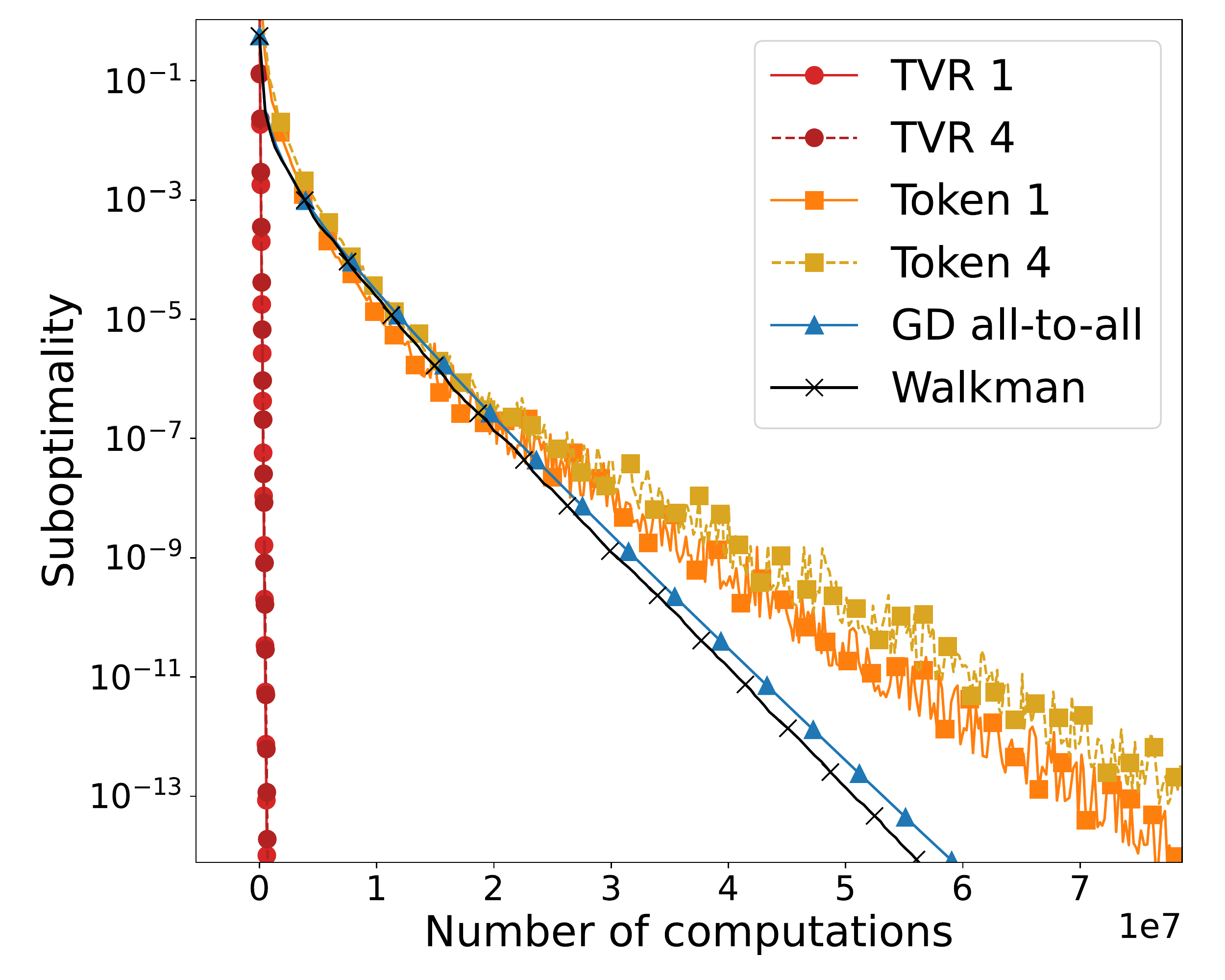}
		
		\centering{(a) Communication (left) and computation (right) complexities for the complete graph}
		
		\centering
		\includegraphics[width=0.45\linewidth]{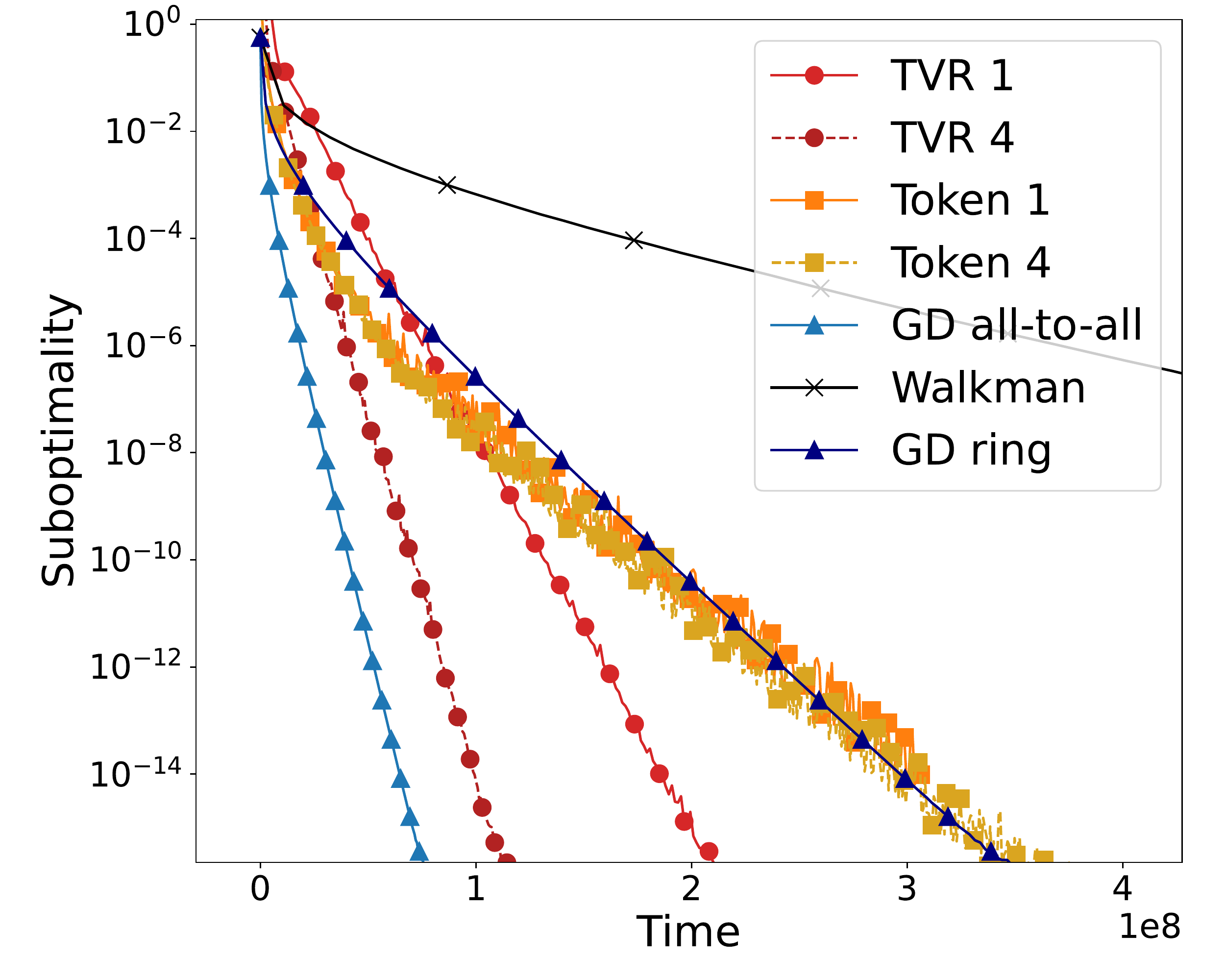}
		\includegraphics[width=0.45\linewidth]{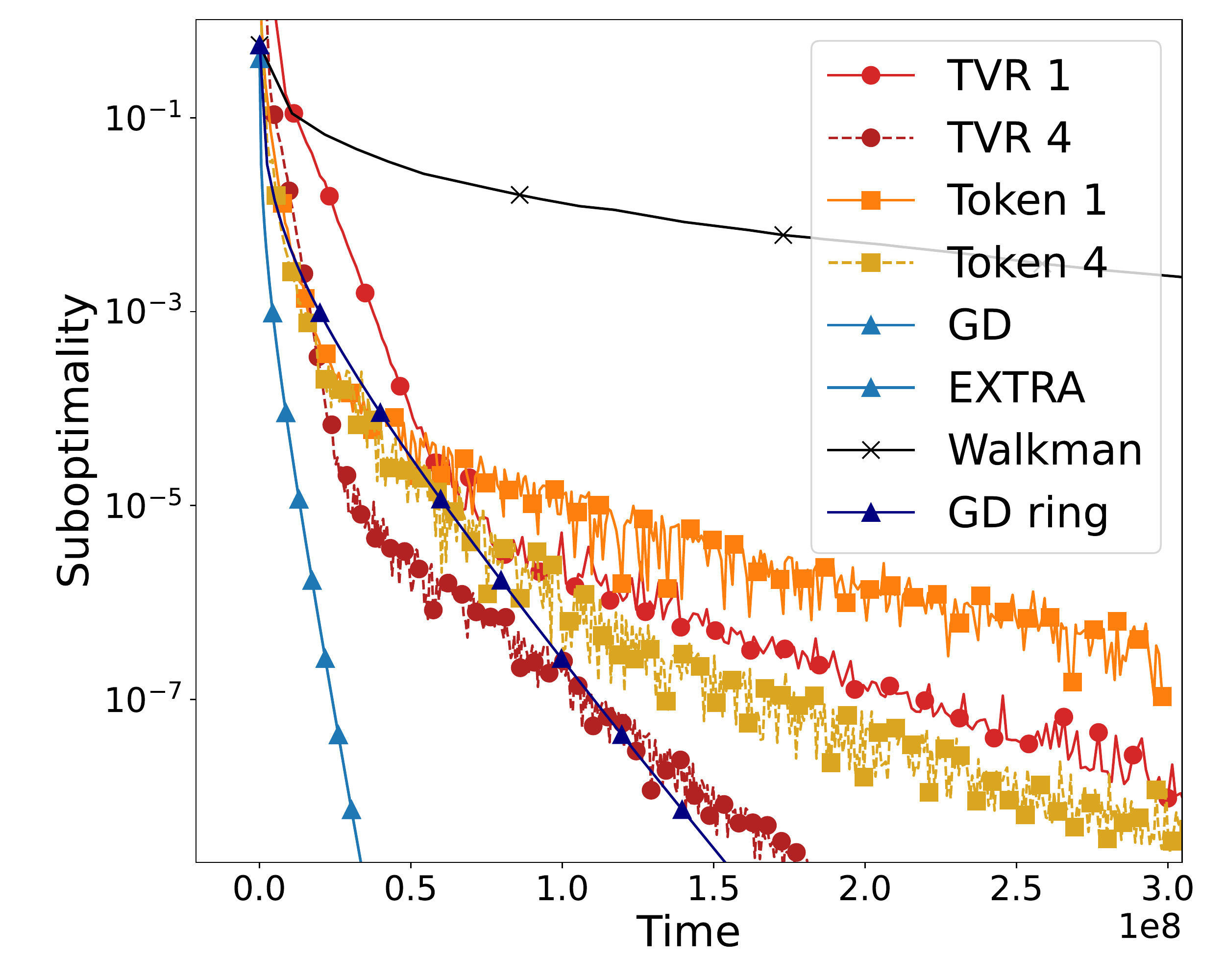}
		
		\centering{(b) Time complexities for the complete (left) and ring (right) graphs}
		
		\caption{Convergence results for a logistic regression task on the RCV1 dataset~\citep{lewis2004rcv1} ($d=47236$), with $\nnodes=20$ nodes, and $\nfs=9841$ samples per node.}
		\label{fig:token_exp} \vspace{-15pt}
	\end{figure}
	
	For the complete graph, all token algorithms have similar communication complexity. This is consistent with our theory, and confirms that using multiple token does not hurt efficiency. We also confirm that the communication complexity of token algorithms is lower than that of all-to-all gradient descent (GD), and comparable to that of the efficient ring GD. Similarly, all batch algorithms have similar computation complexities, with GD and Walkman performing slightly better. TVR is more computationally efficient thanks to the stochastic variance-reduced updates. 
	
	In terms of time, the fastest algorithm is all-to-all gradient descent, since it performs all communications in parallel. Algorithm~\ref{algo:simple_token} is as fast as ring GD, since both algorithms require $O(\nnodes)$ sequential communications, but perform computations in parallel. Walkman is the slowest algorithm in this setting, because it needs to perform both communications and computations sequentially (since it does not use local updates). Note that it would be as fast as the single token algorithm for $\tau_\comm \geq \nfs$. TVR is faster than Algorithm~\ref{algo:simple_token} thanks to variance reduction. We also see that using several tokens speeds up TVR (since the communication time dominates), but not Algorithm~\ref{algo:simple_token} (for which communication and computation times are of the same order). 
	
	For the ring graph, we find that the same token algorithms as for the complete graph are stable (consistently across a wide range of $\nfs$ and $\nnodes$), and so we do not use the skip variant presented in Section~\ref{sec:general_graphs}. We observe similar results to the complete graph case, although using multiple tokens now accelerates Algorithm~\ref{algo:simple_token} since they allow to compensate for the worse graph connectivity. GD and EXTRA have the same rate in this case (since $\kappa > \gamma^{-1})$, and their curves are thus almost indistinguishable. 
	
	\section{Conclusion}
	We have presented a general framework for analyzing token algorithms, and derived several variants such as variance-reduction and acceleration from it. All these token algorithms are competitive with their centralized counterparts in terms of computation and communication complexities. Multiple tokens can be used to increase the level of parallelism, and reduce the communication time. 
	
	We have also discussed a reduction from the general case to the complete communication graph case, in which our results are proven. We claim this reduction leads to efficient algorithms for general graphs, although these algorithms seem to waste communications. An important research direction would be to formalize these claims, and directly analyze versions of these algorithms in which tokens exchange information with all the nodes they visit. This would involve tight analyses of coordinate descent with Markov Chain sampling.

\section{Acknowledgements}
This work was supported by the Swiss National Foundation (SNF) project $200020\_200342$.
	
	\bibliographystyle{plainnat}
	\bibliography{biblio}

	
	\newpage 
	
	\appendix
	
	\section{Revisiting Bregman Coordinate Descent}
	\label{app:bcd}
	The algorithmic core for the non-accelerated methods is based on Bregman Coordinate descent. Yet, the existing theory from~\citet{hendrikx2020dual} was not satisfactory, as it did not allow to consider stochastic communications: the communication block was sampled all at once, which is not possible in the token approach. 
	
	Similarly, in general graphs, we do not know the exact sampling probability $\ptilde$, but an upper bound $p_i$ of it. We thus changed the coordinate descent algorithm to also take that into account.  
	
	We adapt the result from~\citet{hendrikx2020dual} to tackle these two problems in this section. First of all, for two vectors $x,y \in \R^d$, define the Bregman divergence:
	\begin{equation} \label{eq:bd}
	D_h(x,y) = h(x) - h(y) - \nabla h(y)^\top(x - y).
	\end{equation}
	In order not to worry about further regularity assumptions, we assume throughout this paper that all the functions we consider are twice continuously differentiable and strictly convex on ${\rm dom} \, h$, and that $\nabla h
	(x) = \min_y h(y) - x^\top y$ is uniquely defined. This is not very restrictive for typical machine learning objectives. The Bregman divergence has some interesting properties. In particular, $D_h(x,y) \geq 0$ for all $x,y\in \R^d$ as soon as $h$ is convex. For $i \in [d]$, we denote $e_i \in \R^d$ the unit vector corresponding to coordinate $i$. Consider the following Bregman coordinate descent algorithm, in which the iterates are given by:
	\begin{equation}\label{eq:breg_cd}
	x_{t+1} = \arg\min_x \left\{V_{i,t}(x) \hat{=} \frac{\lr}{p_i} \nabla_i f(x_t)^\top \projA x + D_h(x, x_t) \right\},
	\end{equation}
	where $\nabla_i f(x_t) = e_i e_i^\top \nabla f(x_t)$, and $A^\dagger A$ is some projection matrix, which is such that $\projA \nabla f(x) = \nabla f(x)$ for all $x\in \R^d$. An equivalent way to write these iterations is: 
	\begin{equation}\label{eq:breg_cd_eq}
	\nabla h(x_{t+1}) = \nabla h(x_t) - \frac{\lr}{p_i}  \projA \nabla_i f(x_t)
	\end{equation}
	Although we use some $p_i$ for the learning rate, we assume that coordinates are sampled according to another distribution, namely, $\ptilde$. In order to make the training stable, we assume that for all $i,t$, there exist $\deltaproba > 0$ which are such that for all $i$, 
	\begin{equation}
	\ptilde (1 + \deltaproba) = p_i. 
	\end{equation}
	In particular, this means that the $p_i$ are not a proper probability distribution, since they don't sum to one. We denote $\Delta = \sum_{i=1}^d  \ptilde \deltaproba$, which is such that $1 + \Delta$ is normalizing factor of the $p_i$. Similarly, we denote $\delta = \min_{i,t} \deltaproba$. We denote $R_i = (A^\dagger A)_{ii}$ and $R_p = \min_i {R_i^{-1}}{\ptilde}$. We make the following assumptions on $f$ and $h$.
	
	\begin{assumption}[Regularity assumptions] \label{ass:bcd}
		Function $f$ is $\mu$-relatively strongly convex and $L_i$-relatively smooth in the direction $i$ with respect to $h$, meaning that for all $x,y \in \R^d$, and some $v_i$ supported by $e_i$:
		\begin{equation}
		\mu D_h(x,y) \leq D_f(x,y), \ \text{ and } \ D_f(x + v_i, x) \leq L_i D_h(x + v_i, x).
		\end{equation}
		We also make some technical assumptions, that are verified for our problem. In particular, we assume that $h$ and $\projA$ are such that:
		\begin{equation} \label{eq:separability_assumption}
		\nabla_i f(x_t)^\top \projA (x_t - x_{t+1}^{(i)}) = R_i \nabla f(x_t)^\top (x_t - x_{t+1}^{(i)}).
		\end{equation}
	\end{assumption}
	
	Using this assumption, we first prove a technical lemma, which ensures that each step reduces the function value:
	\begin{lemma}[Monotonicity] \label{lemma:monotonicity}
		Under Assumption~\ref{ass:bcd}, if $\lr \leq \frac{\proba}{L_i R_i}$ the iterates of Equation~\eqref{eq:breg_cd} verify for all $i \in [d]$:
		\begin{equation}
		f(x_{t+1}^{(i)}) \leq f(x_t).
		\end{equation}
	\end{lemma}
	\begin{proof}
		Using Assumption~\ref{ass:bcd} , we have that:
		\begin{align*}
			D_h(x_{t+1}^{(i)}, x_t) + D_h(x_t, x_{t+1}^{(i)}) &= \left[\nabla h(x_{t+1}^{(i)}) - \nabla h(x_t)\right]^\top (x_{t+1}^{(i)} - x_t) \\
			&= \frac{\lr}{\proba} \nabla_i f(x_t)^\top \projA (x_t - x_{t+1}^{(i)})\\
			&= \frac{\lr R_i}{\proba} \nabla f(x_t)^\top (x_t - x_{t+1}^{(i)})\\
			&= \frac{\lr R_i}{\proba} \left[D_f(x_{t+1}^{(i)}, x_t) - f(x_{t+1}^{(i)}) + f(x_t)\right]\\
			&\leq \frac{\lr}{\proba} R_i \left[L_i D_h(x_{t+1}^{(i)}, x_t) - f(x_{t+1}^{(i)}) + f(x_t)\right]\\
			&\leq D_h(x_{t+1}^{(i)}, x_t) - \frac{\lr R_i}{\proba}\left[ f(x_t) - f(x_{t+1}^{(i)})\right],
		\end{align*}
		where the last line uses that $\lr \leq \frac{\proba}{L_i R_i}$. In particular:
		\begin{equation*}
			f(x_{t+1}^{(i)}) \leq f(x_t) - D_h(x_t, x_{t+1}^{(i)}) \leq f(x_t).
		\end{equation*}
	\end{proof}
	
	\begin{theorem} \label{thm:bcd}
		Consider two functions $f$ and $h$ that verify Assumption~\ref{ass:bcd}. If the iterates are given by Equation~\eqref{eq:breg_cd}, and denoting for any $x \in {\rm dom}\ h$, $\mathcal{L}_t = (1 + \delta) D_\phi(x, x_{t}) + \frac{\eta_t}{\Rp} [f(x_{t}) - f(x)]$, we obtain for $\eta_t \leq \frac{\proba}{L_i R_i}$: 
		\begin{align*}
			\mathcal{L}_{t+1} \leq \max\left(\frac{1  + \Delta - \eta_t \mu}{1 + \delta}, 1 - \Rp\right) \mathcal{L}_t
		\end{align*}
	\end{theorem}
	
	Note that $\frac{1 + \Delta - \eta \mu}{1 + \delta} \geq 1 - \eta \mu \geq 1 - \frac{\proba}{R_i} \frac{\mu}{L_i} \geq 1 - \frac{(1 + \deltaproba)\mu}{L_i} \Rp$, so the first term generally dominates since we generally have $(1 + \deltaproba)\mu \leq L_i$ unless the condition number is very small, or $\delta_{i,t}$ very large.

	\begin{proof}[Proof of Theorem~\ref{thm:bcd}]
		For any $x \in {\rm dom} h$ We start by writing the 3 points inequality: 
		\begin{align*}
			D_h (x, x_{t+1}) + D_h(x_{t+1}, x_t) - D_h(x, x_t) &= [\nabla h(x_t) - \nabla h(x_{t+1})]^\top (x - x_{t+1})\\
			&= \frac{\lr}{p_i} \nabla_i f(x_t)^\top \projA (x - x_{t+1})\\
			&= \frac{\lr}{p_i} \nabla_i f(x_t)^\top (x - x_t) + \frac{\lr}{p_i} \nabla_i f(x_t)^\top  \projA (x_t - x_{t+1}).
		\end{align*}
		We now multiply everything by $1 + \deltaproba$, leading to: 
		\begin{align*}
			(1 + \deltaproba)\left[D_\hnet (x , x_{t+1}) + D_\hnet(x_{t+1}, x_t) - D_\hnet(x , x_t)\right] &= \frac{\lr (1 + \deltaproba)}{p_i} \nabla_i f(x_t)^\top  \projA (x  - x_t)\\
			&+ \frac{\lr(1 + \deltaproba)}{p_i} \nabla_i f(x_t)^\top \projA (x_t - x_{t+1}),
		\end{align*}
		which we rewrite as:
		\begin{align*}
			(1 + \delta)D_\hnet (x , x_{t+1}) &\leq  (1 + \deltaproba) D_\hnet(x , x_t) + \frac{\lr}{\ptilde} \nabla_i f(x_t)^\top  \projA (x  - x_t)\\
			&+ \frac{\lr}{\ptilde} \nabla_i f(x_t)^\top \projA (x_t - x_{t+1}) - (1 + \deltaproba)D_\hnet(x_{t+1}^{(i)}, x_t),
		\end{align*}
		In particular, when taking an expectation with respect to the sampling distribution, we obtain:
		\begin{align*}
			(1 + \delta)\esp{D_\hnet (x , x_{t+1})} &\leq \sum_{i=1}^d p_i D_\hnet(x , x_t) - \sum_{i=1}^d p_i D_\hnet(x_{t+1}^{(i)}, x_t) \\
			& + \lr \nabla f(x_t)^\top (x  - x_t) + \sum_{i=1}^d \lr \nabla_i f(x_t)^\top \projA (x_t - x_{t+1}^{(i)}).
		\end{align*}
		Using Assumption~\ref{ass:bcd} leads to:
		\begin{align*}
			\lr \nabla_i f(x_t)^\top \projA (x_t - x_{t+1}^{(i)}) &= \lr R_i \nabla f(x_t)^\top (x_t - x_{t+1}^{(i)})\\
			&= \lr R_i \left[D_f(x_{t+1}^{(i)}, x_t) - f(x_{t+1}^{(i)}) + f(x_t)\right]\\
			&\leq \lr R_i \left[L_i D_h(x_{t+1}^{(i)}, x_t) - f(x_{t+1}^{(i)}) + f(x_t)\right]
		\end{align*}
		Using that $\lr \leq \frac{\proba}{L_i R_i}$, we obtain:
		\begin{align*}
			(1 + \delta)\esp{D_h (x , x_{t+1})} &\leq (1 + \Delta) D_h(x , x_t) + \lr \nabla f(x_t)^\top (x  - x_t) + \sum_{i=1}^d \lr R_i \left[ - f(x_{t+1}^{(i)}) + f(x_t)\right].
		\end{align*}
		Note that by monotonicity of the iterations (Lemma~\ref{lemma:monotonicity}), $f(x_{t+1}^{(i)}) \leq f(x_t)$ for all $t,i$, and so, with $R_p = \min_i R_i^{-1} \ptilde$:
		\begin{equation}
		\sum_{i=1}^d \frac{R_i}{\ptilde} \ptilde \left[ - f(x_{t+1}^{(i)}) + f(x_t)\right] \leq \Rp^{-1} \left[ - \sum_{i=1}^d \ptilde f(x_{t+1}^{(i)}) + f(x_t)\right] \leq - \Rp^{-1}\left(\esp{f(x_{t+1})} - f(x_t)\right)
		\end{equation}
		Finally, the $\mu$-relative strong convexity of $f$ gives:
		\begin{equation}
		\nabla f(x_t)^\top (x  - x_t) = - \left[f(x_t) - f(x )\right] - D_f(x , x_t) \leq - \mu D_h(x , x_t) - f(x_t) + f(x )
		\end{equation}
		Combining everything leads to: 
		\begin{align*}
			(1 + \delta)\esp{D_h (x , x_{t+1})} + &\frac{\lr}{\Rp} \left[f(x_{t+1} - f(x )\right] \\
			&\leq \frac{1 + \Delta - \eta \mu}{1 + \delta} (1 + \delta) D_h(x , x_t) + \frac{\lr}{\Rp} \left(1 - \Rp\right) \left[f(x_t) - f(x )\right]    
		\end{align*}
		
	\end{proof}
	
	\section{Convergence results}
	\label{app:simple_token}
	
	\subsection{Introducing the problem}
	
	In this section, we introduce the consensus problem derived from the conceptual graph, and take its dual formulation. This section follows the same framework as~\citet{hendrikx2020dual}. Denoting $\sigmatilde = \frac{\nnodes\sigma}{\nnodes + \ntokens}$, this problem writes: 
	\begin{equation}\label{eq:full_consensus_pb}
	\underset{\ \forall i, \ \forall j, \ \theta^{(ij)} = u^{(i)}, \ \forall k, \ u^{(i)} = v^{(j)}}{\min_{\theta \in \R^{n \times d}, \ u \in \R^{n\times d}, \ v \in \R^d}} \sum_{i=1}^\nnodes \sum_{j=1}^\nfs f_{ij}(\theta^{(ij)}) + \frac{\sigmatilde}{2} \|u^{(i)}\|^2 + \sum_{k=1}^\ntokens \frac{\sigmatilde}{2}\|v^{(k)}\|^2.
	\end{equation}
	Introducing Lagrangian multipliers $y$ for each consensus constraint in the virtual part of the graph (between center nodes $u^{(i)}$ and computation nodes $\theta^{(ij)}$), and multipliers $x$ for the communication part of the graph, (between center nodes $u^{(i)}$ and tokens $v^{(k)}$), the dual problem writes:
	\begin{equation} \label{eq:full_dual_problem}
	\min_{x \in \R^{\nnodes \ntokens d},\  y \in \R^{\nnodes \nfs d}}  q_A(x,y) + \sum_{i=1}^n \sum_{j=1}^m f_{ij}^*((Ay)^{(ij)}), \hbox{ with } q_A(x,y) \triangleq \frac{1}{2\sigmatilde}(x, y)^\top A^\top A (x, y),
	\end{equation}
	where $A \in \R^{(\nnodes(\nfs + 1) + \ntokens) d \times \nnodes(\ntokens + \nfs)d}$ is the (weighted) incidence matrix of the augmented conceptual graph, which is such that $A e_{\ell_1 \ell_2} = \mu_{\ell_1 \ell_2} (e_{\ell_1} - e_{\ell_2})$, for any two nodes $\ell_1$ and $\ell_2$ (and an arbitrary orientation), where $e_{\ell_1 \ell_2}$ is the unit vector corresponding to edge $(\ell_1, \ell_2)$, and $e_{\ell_1}$, $e_{\ell_2}$ are the unit vectors corresponding to nodes $\ell_1$ and $\ell_2$. Note that we abused notations and wrote $(Ay)^{(ij)}$ instead of $(A(x,y))^{(ij)}$, since $(A(x,y))^{(ij)}$ does not actually depend on $x$. 
	
	Matrix $A$ has a very special structure, since it is the incidence matrix of a tripartite graph between computation nodes, communication nodes, and token nodes. We now have to choose the weights of matrix $A$. For communication edges (between communication nodes and tokens), we make the simple choice $\mu_{ik} = 1$. For computation edges (between communication nodes and computation nodes), we choose:
	\begin{equation}
	\mu_{ij}^2 = \alpha L_{ij}, \ \text{ with } \ \alpha = \frac{2\sigmatilde \ntokens}{\kappa_s},
	\end{equation}
	where $L_{ij}$ is the smoothness of function $f_{ij}$ and $\kappa_s = \max_i 1 + \sum_{j=1}^\nfs L_{ij} / \sigmatilde$. These choices follow from~\citet{hendrikx2020dual}, where we used that the ``communication part'' of the conceptual graph used to derive the token algorithm is quite specific (complete bipartite graph), and so $\lambda_{\min}(A_{\rm comm}^\top A_{\rm comm}) = \ntokens$, the number of tokens, as long as $\ntokens \leq \nnodes$~\citep{brouwer2011spectra}.
	
	\subsection{The TVR algorithm}
	\label{app:tvr}
	We now proceed to the derivation of the DVR algorithm, and to proving its theoretical guarantees. To achieve this, we use the Bregman coordinate descent iterations defined in the previous section, which are of the form:
	\begin{equation}\label{eq:breg_cd_tvr}
	(x,y)_{t+1} = \arg\min_{x,y} \frac{\lr}{p_i} \nabla_i \left[q_A((x,y)_t) + F^*(Ay_t) \right]^\top \projA x + D_h((x,y), (x,y)_t),
	\end{equation}
	where $F^*(\lambda) = \sum_{i,j} f_{ij}^*(\lambda^{(ij)})$. We now need to define the reference function $h$, which we define, following~\citet{hendrikx2020dual}, as $h(x,y) = h_x(x) + h_y(y)$, with:
	\begin{equation} \label{eq:reference_function}
	\forall i,j, \ h_{y}^{(ij)}(y^{(ij)}) = \frac{L_{ij}}{\mu_{ij}^2}f_{ij}^*(\mu_{ij} y^{(ij)}), \ \text{ and } \ h_x(x) = \frac{1}{2} \|x\|^2_{A_\comm^\dagger A_\comm},
	\end{equation}
	where $A_\comm \in \R^{(\nnodes + \ntokens) d \times (\nnodes (\nfs + \ntokens)}$ is the restriction of $A$ to communication nodes (and tokens). In order to avoid notations clutter, we slightly abuse notations and use $\projA$ instead in the remainder of this section.
	
	\begin{algorithm}
		\caption{Token Variance Reduced $(z_0)$}
		\label{algo:tvr}
		\begin{algorithmic}[1]
			\STATE $\tilde{\sigma} = \sigma \frac{\nnodes}{\nnodes+\ntokens}$ $\alpha = \frac{2 \ntokens \sigmatilde}{\kappa_s}$, $\pcomp = \left(1 + \frac{\kappa_s}{\nfs - 1 + \kappa_s}\right)^{-1}$, $\pcomm = 1 - \pcomp$, 
			\STATE $\eta = \min\left( \frac{\tilde{\sigma} p_\comm}{2\nnodes \ntokens}, \frac{\pcomp}{ \alpha \nnodes (\nfs - 1 + \kappa_s)} \right)$, $\rho_\comm = \frac{\nnodes \ntokens \eta}{p_\comm\tilde{\sigma}}$, $\rho_{\comp} = \frac{\nfs \nnodes \alpha \eta}{\pcomp}$. \COMMENT{Init}
			\STATE $\forall i \in [\nnodes]$, $\theta_0^{(i)} = - \sum_{j=1}^\nfs \nabla f_{ij}(z_0^{(i,j)}) / \tilde{\sigma}$;   \ \ $\forall k \in [\ntokens]$, $\theta_0^{\token, k} = 0$. \COMMENT{$z_0$ is arbitrary but not $\theta_0$.}
			\FOR[Run for $T$ iterations]{$t=0$ to $T-1$}
			\IF{ communication step (with probability $p_\comm$)}
			\STATE Pick $i \sim \mathcal{U}([\nnodes]), \ k \sim \mathcal{U}([\ntokens])$ \COMMENT{Choose next node and token uniformly at random}
			\STATE $\theta_{t+1}^{\token, k} = \theta_t^{\token, k} - \rho_{\comm}  (\theta_t^{\token, k} - \theta_t^{(i)})$ \COMMENT{Token update}
			\STATE $\theta_{t+1}^{(i)} = \theta_t^{(i)} + \rho_{\comm}(\theta_t^{\token, k} - \theta_t^{(i)})$ \COMMENT{Local node update}
			\ELSE 
			\STATE Pick $i \sim \mathcal{U}([\nnodes])$, $j \sim \mathcal{U}([\nfs])$ \COMMENT{Choose one node and data point at random}
			\STATE $z_{t+1}^{(i,j)} = \left(1 - \rho_{\comp}\right)z_t^{(i,j)} + \rho_{\comp} \theta_t^{(i)}$ \COMMENT{Virtual node update}
			\STATE $\theta_{t+1}^{(i)} = \theta_t^{(i)} - \frac{1}{\tilde{\sigma}}\left(\nabla f_{ij}(z_{t+1}^{(i,j)}) - \nabla f_{ij}(z_{t}^{(i,j)})\right)$ \COMMENT{Local update using $f_{ij}$}
			\ENDIF
			\ENDFOR
			\STATE \textbf{return} $\theta_K$
		\end{algorithmic}
	\end{algorithm}
	
	For the computation part, the algorithm can be recovered by following the exact same steps as \citet[Section 2]{hendrikx2020dual}, which are themselves inspired by the dual-free updates from~\citet{lan2017optimal}.
	
	For the communication part, there is a small difference in the fact that we do not sample all coordinates at once anymore. Instead, we sample edges of the communication graph one by one. The other difference is that communication edges are not between node $i$ and node $k$ anymore, but between node $i$ and \emph{token} $k$. In particular, the communication step writes:
	\begin{equation}
	\projA x_{t+1} = \projA x_t - \frac{\eta}{\sigmatilde p_{ik}} \projA e_{ij} e_{ik}^\top A^\top A (x_t, y_t).
	\end{equation} 
	Multiplying by $A$ on the left and defining $\theta_t = A_\comm (x_t,y_t)$ leads to: 
	\begin{equation}
	\theta_{t+1} = \theta_t - \frac{\eta}{\sigmatilde p_{ik}} A_\comm e_{ik} e_{ik}^\top A^\top \theta_t,
	\end{equation}
	where we used the fact that the update only affects communication nodes. Therefore, when we sample an edge between node $i$ and token $k$, this leads to:
	\begin{equation}
	\theta_{t+1} = \theta_t - \frac{\eta \nnodes \ntokens}{\pcomm \sigmatilde} W_{ik}\theta_t,
	\end{equation}
	where $W_{ik} = (e_i - e_k)(e_{i} - e_{k})^\top$, thus leading to the form obtained in Algorithm~\ref{algo:tvr}. From there, we can prove the following convergence theorem, from which all the non-accelerated results from the main text are derived. 
	
	\begin{theorem} \label{thm:tvr_full}
		The iterates of Algorithm~\ref{algo:tvr} verify:
		\begin{equation}
		\|\theta_t - \theta_\star\|^2 \leq \left(1 - \frac{\eta \ntokens \sigmatilde}{\kappa_s}\right)^t C_0, 
		\end{equation}
		where $C_0 = 2\sigmatilde \mathcal{L}_0$, with $\mathcal{L}_0$ the Lyapunov function from Theorem~\ref{thm:bcd} instantiated on the dual problem, which thus depends on the initialization. In particular, ignoring log factors in $C_0$, if $p_\comm = p_\comp = \frac{1}{2}$ and for any $\varepsilon > 0$, 
		\begin{equation}
		\ncomps = O((\nfs + \kappa_s)\log\varepsilon^{-1}) \ \text{ and } \  \ncomms = O(\nnodes \kappa_s \log\varepsilon^{-1})
		\end{equation}
		are required in total in order to obtain $\|\theta_t - \theta_\star\|^2 \leq \varepsilon$.
		
	\end{theorem}
	
	The proof of this algorithm follows several steps, that we detail in the next subsections. We first show that the objective function defined in Equation~\eqref{eq:full_consensus_pb}, together with the reference function $h$ defined in Equation~\eqref{eq:reference_function} satisfies Assumption~\ref{ass:bcd}, so that we can apply Theorem~\ref{thm:bcd}. Then, we show how to choose the remaining parameters (and in particular $\pcomm$ and $\pcomp$) optimally, and evaluate the rate in terms of constants of the problem (number of nodes, number of tokens, smoothness and strong convexity of the local functions...).
	
	\subsection{Verifying Assumption~\ref{ass:bcd}}
	We start by showing that Assumption~\ref{ass:bcd} is verified in this case, which includes three parts: Equation~\eqref{eq:separability_assumption}, relative strong convexity, and directional relative smoothness. 
	
	\subsubsection{Verifying Structural assumptions.} We first verify that the updates of our problem verify Equation~\eqref{eq:separability_assumption} from Assumption~\ref{ass:bcd}.
	
	\textbf{1 - Computation coordinates.} Computation coordinates corresponding to the edge between computation and communication subnodes in Figure~\ref{fig:simple_token}. In particular, the graph becomes disconnected if they are removed, thus implying that $\projA e_i = e_i$. In particular, 
	\begin{equation}
	\nabla h(x_{t+1}) = \nabla h(x_t) - \frac{\lr}{p_i} \nabla_i f(x_t).
	\end{equation}
	Yet, for our specific choice of $h$ (which is such that $\hcomp(x) = \sum_i \hcomp^{(i)}(x^{(i)})$), this implies that $x_{t+1}^{(i)} - x_t = v_i$ for some $v_i$ that only has support on coordinate $i$, and in particular:
	\begin{equation}
	\nabla_i f(x_t)^\top \projA (x_t - x_{t+1}^{(i)}) = R_i \nabla_i f(x_t)^\top v_i = R_i \nabla f(x_t)^\top v_i  = R_i \nabla f(x_t)^\top (x_t - x_{t+1}^{(i)})
	\end{equation}
	
	\textbf{2 - Network coordinates.} Network coordinates have a different structure. In this case, the reference function $\hcomm$ is the quadratic form induced by $\projA$, which facilitates analysis. The updates write: 
	\begin{equation}
	\projA x_{t+1} = \projA x_t - \projA \frac{\lr}{p_i} \nabla_i f(x_t)
	\end{equation}
	Although $\hcomm$ is not separable, we can leverage the presence of $\projA$ to write:
	\begin{equation*}
		\nabla_i f(x_t)^\top \projA (x_t - x_{t+1}^{(i)}) = \frac{\eta_t}{p_i} \nabla_i f(x_t)^\top \projA \nabla_i f(x_t) = \frac{\eta_t}{p_i}R_i \nabla f(x_t)^\top \nabla_i f(x_t).
	\end{equation*}
	We then use that $\nabla f(x_t) = \projA \nabla f(x_t)$, leading to:
	\begin{equation*}
		\nabla_i f(x_t)^\top \projA (x_t - x_{t+1}^{(i)}) = \frac{\eta_t}{p_i} R_i \nabla f(x_t)^\top \projA \nabla_i f(x_t) = R_i \nabla f(x_t)^\top \projA (x_t - x_{t+1}^{(i)}).
	\end{equation*}
	Now that we have proven that $h$ and $f$ verify the structural assumptions given by Equation~\eqref{eq:separability_assumption}, it remains to evaluate the relative strong convexity and directional smoothness constants $\mu$ and $L_i$. 
	
	\subsubsection{Relative Strong Convexity.} Since the structure of the dual problem and the reference function $h$ are the same, we directly have from~\citet[Appendix B.1]{hendrikx2020dual} that the relative strong convexity constant is equal to
	\begin{equation}
	\mu = \frac{\alpha}{2} = \frac{\ntokens}{\sigmatilde + \sum_{j=1}^\nfs L_{ij}},
	\end{equation}
	since the smallest eigenvalue of the communication graph is equal to $\ntokens$ (the number of tokens) in this case. 
	
	\subsubsection{Directional Relative Smoothness.} We now evaluate the relative smoothness constants.
	
	\textbf{1 - Computation edges.} In the computation case, similarly to strong convexity, we directly get from~\citet[Appendix B.1]{hendrikx2020dual} that $\smooth_{ij}$, the relative directional smoothness for virtual node $i,j$ (or just $\smooth_i$ if there is only one virtual node), can be obtained as:
	\begin{equation}
	\smooth_{ij} = \alpha \left(1 + \frac{L_{ij}}{\sigma_i}\right),
	\end{equation}
	Plugging in the value of $\alpha$, this leads to: 
	\begin{equation}
	\smooth_{ij} = \frac{2\ntokens}{\sigmatilde} \frac{\sigmatilde + L_{ij}}{\sigmatilde + \sum_{j=1}^\nfs L_{ij}}
	\end{equation}
	
	\textbf{2 - Communication edges.} In this case, we cannot use the results from DVR directly because the sampling of communication coordinates is different. While DVR sampled all communication edges at once, we only sample one at each step. In this case, we have that the directional relative smoothness is equal to: 
	\begin{equation}
	D_f(x + \Delta_{uv}, x) = \|\Delta_{uv}\|^2_{A^\top \Sigma A}= \mu_{uv}^2(\sigma_u^{-1} + \sigma_v^{-1})\|\Delta_{uv}\|^2 = \frac{\mu_{uv}^2(\sigma_u^{-1} + \sigma_v^{-1})}{e_{uv}^\top \projA e_{uv}} \|\Delta_{uv}\|^2_{\projA}.
	\end{equation}
	In particular, for communication edges, and with the choice that $\mu_{uv}^2 = 1$ and $\sigma_u = \sigma_v = \sigmatilde$:
	\begin{equation}
	D_f(x + \Delta_{uv}, x) \leq \smooth_{uv} D_h(x + \Delta_{uv}, x), \ \text{ with } \ \smooth_{uv} = \frac{2}{\sigmatilde R_{uv}}
	\end{equation}
	
	\subsection{Convergence guarantees}
	\label{app:cv_tvr}
	We have shown in the previous subsection that we can apply Theorem~\ref{thm:bcd} to obtain convergence guarantees for our token algorithms. For the communication edges, the step-size constraint leads to:
	\begin{equation}
	\lr \leq \frac{p_{uv}}{R_{uv} \smooth_{uv}} = \frac{p_\comm \sigmatilde}{2\nnodes \ntokens}
	\end{equation}
	For the computation edges, we can set (as in the DVR article) $p_{ij} \propto 1 + L_{ij} / \sigmatilde$. In particular, the normalizing factor is equal to $\sum_{i=1}^n\sum_{j=1}^\nfs 1 + L_{ij} / \sigmatilde = \nnodes(\nfs + \sum_{j=1}^\nfs L_{ij} / \sigmatilde) = \nnodes(\nfs - 1 + \kappa_s)$, where we recall that $\kappa_s = 1 + \sum_{j=1}^\nfs L_{ij} / \sigmatilde$. Therefore, we obtain: 
	\begin{equation}
	\lr \leq \frac{p_{ij}}{\smooth_{ij}} \leq \frac{p_\comp \sigmatilde \kappa_s}{2\nnodes \ntokens (\nfs - 1 + \kappa_s)}.   
	\end{equation}
	We want to balance $\pcomm$ and $\pcomp$ such that these two constraints match, leading to:
	\begin{equation}
	\pcomm = \frac{\kappa_s}{\nfs - 1 + \kappa_s} \pcomp.
	\end{equation}
	Since $\pcomm = 1 - \pcomp$, this leads to:
	\begin{equation}
	\pcomp = \left(1 + \frac{\kappa_s}{\nfs - 1 + \kappa_s}\right)^{-1}.
	\end{equation}
	
	Assuming $\delta \geq 0$ and $\Delta < \frac{\lr\mu}{2}$, the rate of convergence of Algorithm~\ref{algo:tvr} is $\rho = \lr \mu - \Delta \geq \lr \mu /2$. In particular, we directly have that the computation complexity is equal to:
	\begin{equation}
	\frac{\pcomp}{\rho} = 4 (\nfs - 1 + \kappa_s).
	\end{equation}
	Similarly, the communication complexity is equal to:
	\begin{equation}
	\frac{\pcomm}{\rho} = 2\nnodes\kappa_s.
	\end{equation}

	\subsection{Special cases}
	\label{app:multiple_tokens}
	
	\subsubsection{Complete graphs}
	All the theorems in the main paper are actually direct corollaries of Theorem~\ref{thm:tvr_full}. We provide below how they can be derived in each case.
	
	\textbf{Theorem~\ref{thm:token_vr}}: We apply Theorem~\ref{thm:tvr_full} with $\delta = 0$ (since the graph is complete, so we know the true sampling distribution).
	
	\textbf{Theorem~\ref{thm:multiple_tokens}}: When $\nfs = 1$, all the derivations remain the same, but we now have that $\kappa_s = 1 + L_i / \sigmatilde = \kappa$, and so the computation complexity is equal to $\nfs - 1 + \kappa_s = \kappa$. 
	
	\textbf{Theorem~\ref{thm:simple_token}}: This result can be recovered by simply taking $K=1$.
	
	\subsubsection{General graphs.}
	Consider that the transitions between nodes are ruled by matrix $W$, which is such that $\|W^t \pi_0 - \pi_\star\|_\infty \leq C (1 - \gamma)^t$ for any starting distribution $\pi_0$, with $C > 0$ a constant, $\pi_\star$ the stationary distribution of the random walk, and $\gamma > 0$ a constant which can be interpreted as the inverse of the mixing time of the Markov Chain with transition matrix $W$. This is true as long as the underlying Markov Chain is irreducible and aperiodic. In this case, then after $O(\gamma^{-1} \log(C/(\eta\mu))$ steps, we have that for all $i$: 
	\begin{equation}
	|\ptilde - (\pi_\star)_i| \leq \frac{\eta\mu}{4},
	\end{equation}
	so in particular by taking $\proba = (\pi_\star)_i + \frac{\eta\mu}{4}$ satisfies $\ptilde(1 + \deltaproba) = \proba$, with $0 \leq \deltaproba \leq \frac{\eta\mu}{2}$. Then, using Theorem~\ref{thm:bcd}, we recover the same result as in Theorem~\ref{thm:tvr_full}, with $\eta\mu$ replaced by $\eta \mu - \Delta \geq \frac{\eta\mu}{2}$. 
	
	Note that the value of $\eta$ depends on $\Delta$, which itself depends on $\eta$, so the above derivations technically result in a circular argument. To avoid this, one can simply use a slightly different $\tilde{\eta} = \min_i \frac{(\pi_\star)_i}{L_i R_i} \leq \eta$ to set the number of token jumps. In practice, we do not need to precisely evaluate these log factors, and taking $\tilde{C} \gamma^{-1}$ jumps with a small constant $\tilde{C}$ is enough.

	\section{Acceleration}
	\label{app:acceleration}
	In the accelerated case, the theory does not follow directly from Theorem~\ref{thm:tvr_full}, since the algorithmic core is different. Indeed, we use a variant of Accelerated Proximal Coordinate Gradient~\citep{lin2015accelerated} instead of Bregman coordinate descent on the dual formulation. Yet, we can directly reuse the convergence results for ADFS~\citep[Theorem 1]{hendrikx2019accelerated}, which we (informally) state below:
	\begin{theorem}
		ADFS has iteration complexity $O(\rho \log \varepsilon^{-1})$, with 
		\begin{equation}
		\rho^2 \leq \min_{k \ell} \frac{\lambda_{\min}^+(A^\top \Sigma^{-1} A)}{\Sigma_{kk}^{-1} + \Sigma_{ll}^{-1}} \frac{p_{k\ell}^2}{\mu_{k\ell}^2 R_{k\ell}}.
		\end{equation}
	\end{theorem}
	
	For our problem (conceptual graph), we obtain the following values for the parameters involved in the computation of $\rho$ when a communication edge $(k, \ell)$ between a node and the token is sampled:
	\begin{align*}
		&\lambda_{\min}^+(A^\top \Sigma^{-1} A) = \frac{\ntokens}{2 \sigmatilde \kappa_s}\\
		&\Sigma_{kk}^{-1} = \Sigma_{\ell\ell}^{-1} = \sigmatilde^{-1}\\
		&p_{k\ell} = \frac{\pcomm}{\nnodes\ntokens}\\
		&\mu_{k\ell} = 1\\
		&R_{k\ell} = \frac{1}{\ntokens}.
	\end{align*}
	In the end, this leads to 
	\begin{equation}
	\rho_\comm^2 = \frac{\pcomm^2}{4n^2\kappa_s}. 
	\end{equation}
	Similarly, we have (just like in the ADFS paper, since the computation part of the graph is the same):
	\begin{equation}
	\rho_\comp^2 = \frac{\pcomp^2}{2n^2 (m + \sqrt{m\kappa_s})^2}.
	\end{equation}
	We now fix $p_\comm$ and $p_\comp$ so that $\rho_\comm = \rho_\comp$, similarly to Section~\ref{app:cv_tvr}. This leads to 
	\begin{equation}
	p_\comm = \frac{2\kappa_s}{m + \sqrt{m\kappa_s}} p_\comp,
	\end{equation}
	and so the communication and computation complexities are respectively:
	\begin{equation*}
		\frac{p_\comm}{\rho_\comm} = 2n \kappa_s, \text{ and }  \frac{p_\comp}{\rho_\comp} = \sqrt{2}n (m + \sqrt{m\kappa_s}).
	\end{equation*}
	Theorem~\ref{thm:tavr} is obtained by expressing these complexities in terms of per-node and per-token quantities. 
	
	\section{Experiments}
	\label{app:experiments}
	
	For the experiments, we use the same setting as~\citet{hendrikx2020dual}, meaning that we solve the following logistic regression problem:
	\begin{equation}
	\min_{\theta \in \R^d}  \sum_{i=1}^n \left[\frac{\sigma}{2} \|\theta\|^2 + \sum_{j=1}^m \frac{1}{m}\log(1 + \exp(-y_{ij} X_{ij}^\top \theta)) \right],
	\end{equation}
	where the pairs $(X_{ij}, y_{ij}) \in \R^d \times \{-1, 1\}$ are taken from the RCV1 dataset, which we downloaded from \url{https://www.csie.ntu.edu.tw/~cjlin/libsvmtools/datasets/binary.html}. We choose the regularization parameter as $\sigma = 10^{-5}$. All experiments were run on a standard laptop, but using MPI to communicate between nodes. 
	
	\textbf{Time.} We choose to report ideal times: to get the execution time of an algorithm, we compute the minimum time it takes to execute its sequence of updates, given fixed communication and computation delays $\tau_\comm$ and $\tau_\comp$. More specifically, we draw a sequence of actions $S$, and denote $S_\ell$ the $\ell$-th action from this sequence, and $T_i(\ell)$ the time at which node $i$ finishes executing update $\ell$. All nodes start from $T_i(0) = 0$.
	\begin{itemize}
		\item If $S_\ell$ is a local computation at node $i$, then node $i$ increases its local time by $\tau_\comp$, \emph{i.e.}, $T_i(\ell) = T_i(\ell - 1) + \tau_\comp$. For $j\neq i$, $T_j(\ell) = T_j(\ell - 1)$.
		\item If $S_\ell$ is a token jump from node $j$ to node $i$, then $T_i(\ell) = \max(T_i(\ell - 1), T_j(\ell - 1) + \tau_\comm)$. For $k\neq i$, $T_k(\ell) = T_k(\ell - 1)$. 
	\end{itemize}
	The sequence $S_\ell$ is implemented by simply sharing a random seed between nodes. Token algorithms could also be implemented without executing this shared schedule, but this would not strictly correspond to Algorithm~\ref{algo:tvr} since the sampling of the edges would not be \emph{i.i.d.}
	
	\textbf{Batch smoothness.} Since the smoothness of the full functions $f_i$ is hard to compute, we approximated it by taking $L_{\rm batch} = 0.02 \times \max_{ij} L_{ij}$. Note that, following~\citet{hendrikx2020dual}, we implemented TVR with this batch smoothness instead of $\sum_{j} L_{ij}$ (which corresponds to taking $\alpha = 2 \sigmatilde K / \kappa$ instead of $\alpha = 2 \sigmatilde K / \kappa_s$). In particular, the communication complexity of TVR is thus proportional to $\kappa$ (similarly to that of Algorithm~\ref{algo:simple_token}) instead of $\kappa_s$. We proved Theorem~\ref{thm:token_vr} with $\kappa_s$ since it is simpler and less restrictive. 
	
	\textbf{Code.} We provide the code used to run the experiments from Figure~\ref{fig:token_exp} in supplementary material. All algorithms are coded in Python, using MPI for communications. This code has not been optimized for efficiency, but rather aims at providing an actual implementation of token algorithms that can be used out of the box. Due to the similarities between algorithms, we based this code on the code in the supplementary material from~\citet{hendrikx2020dual}. 
	
\end{document}